\documentclass[11pt]{amsart}
\usepackage[english]{babel}
\usepackage{amsmath}
\usepackage{amsthm}
\usepackage{amssymb}
\usepackage{amscd}
\usepackage{latexsym}
\usepackage[all]{xy}
\xyoption{all}
\newtheorem{lemma}{Lemma}[section]
\newtheorem{theorem}[lemma]{Theorem}
\newtheorem{corollary}[lemma]{Corollary}
\newtheorem{proposition}[lemma]{Proposition}
\newtheorem{remark}[lemma]{Remark}

\theoremstyle{definition}
\newtheorem*{definition}{Definition}
\usepackage{anysize}
\marginsize{2.5cm}{2.5cm}{2.5cm}{3.4cm}

\begin{document}
\title{Submonoids of the Formal Power Series}
\author{Edgar Enochs}
\address{Department of Mathematics, University of Kentucky}
\email{e.enochs@uky.edu}

\author{Overtoun Jenda}
\address{Department of Mathematics and Statistics, Auburn University }
\email{jendaov@auburn.edu}

\author{Furuzan Ozbek}
\address{Department of Mathematics and Statistics, Auburn University }
\email{fzo0005@auburn.edu}

\date{}

\subjclass[2000]{16W60, 13F25}

{\keywords{Formal power series, monoid, strongly closed submonoid, invertible formal power series.}

\begin{abstract}
Formal power series come up in several areas such as formal language theory , algebraic and enumerative combinatorics,  semigroup theory, number theory etc. (\cite{br}, \cite{dz}, \cite{drs}, \cite{mm}, \cite{hw}). This paper focuses on the the subset $x R[[x]]$ consisting of formal power series with zero constant term. This subset forms a monoid with the composition operation on series. We classify the sets $T$ of strictly positive integers for which the set of formal power series $$R[[x^{T}]]=\{\sum_{t\in T} a_{t}x^{t} \, | \, \, where \, \, a_t \in R \}$$
forms a monoid with composition as the operation. We prove that in order for $R[[x^{T}]]$ to be a monoid, $T$ itself has to be a submonoid of $(\mathbb{N}, \cdot)$. Unfortunately, this condition is not enough to guarantee the desired result. But if a monoid is \emph{strongly closed}, then we get the desired result. We also consider an analogous problem for power series in several variables.

\end{abstract}

\maketitle

\section{Introduction}

Let $R$ be a ring with identity and $R[[x]]$ denote the ring of formal power series in the indeterminant $x$ with coefficients in $R$. Throughout the paper we will denote the natural numbers with $\mathbf{N}$ and the non-zero natural numbers with $\mathbf{N}^*$. Then $x R[[x]]= R[[x^{\mathbf{N}^{*}}]]$  will denote the series with zero constant term. The order of a formal power series $f(x)=\sum a_n x^{n}$ is the smallest power of $x$ appearing in the infinite sum and denoted by $\omega(f(x))$ (i.e. $\omega(f(x))=n$ if $a_n\neq 0$ and if $a_i=0$ for all $0 \leq i < n$, and $\omega(0)=\infty$). Note that $R[[x]]$ forms a group with respect to addition but not with composition. In this paper, we focus on $xR[[x]]$ which forms a monoid with composition as the operation. Composition is defined as usual (i.e. $f,g \in x R[[x]]$, $ f \circ g (x) = f(g(x))$). The identity of $x R[[x]]$ is $x$ and its subset consisting of formal power series with $f'(0)$ a unit of $R$ forms a group.

There are some obvious submonoids of $x R[[x]]$. For example if $k \geq 2$, the set
$$U=\{f(x) \in x R[[x]] \, \,  | \, \, f(x)=a_1 x+ a_k x^k + a_{k+1}x^{k+1}+ \, ... \, \} $$
forms a submonoid. If $d \geq 1$, another useful submonoid is
$$V=\{f(x) \in x R[[x]] \, \,  | \, \, f(x)=a_1 x+ a_2 x^{1+d} + a_{3}x^{1+2d}+ \, ... \, \} $$
Note that in each case we can specify the submonoid by considering a subset $T \subset \mathbf{N}^{*}$  and then taking a series in $x^{t}$ with $t \in T$. This set of series will be denoted $R[[x^{T}]]$. Our object is to find all $T \subset \mathbf{N}^{*}$ such that $R[[x^T]]$ is a submonoid of $ R[[x^{\mathbf{N}^{*}}]]$. In order for $R[[x^T]]$ to be a submonoid of $ R[[x^{\mathbf{N}^{*}}]]$ there are two necessary conditions. Since we want $x \in R[[x^T]]$, we must have $1 \in T$. Also if $s,t \in T$ then $x^s \circ x^t = x^{st}$ must be in $R[[x^{T}]]$, so we must have $st \in T$. That is $T$ must be a submonoid of the multiplicative monoid $\mathbf{N}^{*}$.

\begin{remark} If $R[[x^T]]$ is a submonoid of $ R[[x^{\mathbf{N}^{*}}]]$ then $T$ is a submonoid of $\mathbf{N}^{*}$ with respect to multiplication.
\end{remark}

The above remark gives a necessary condition for our goal but it is not sufficient. To see this consider the multiplicative submonoid $T=\{ 1,2,4,6,...\}$ of $\mathbf{N}^*$. Note that the corresponding series $R[[x^T]]$ is not a submonoid, since $x^2 \circ (x+x^2) \notin R[[x^T]]$. So we need a stronger condition on $T$ to guarantee that $R[[x^T]]$ is a submonoid of $R[[x^{\mathbf{N}^{*}}]]$.

\section{Main Result}
In the previous section we concluded that $T$ being a multiplicative submonoid is not sufficient for our goal. The following theorem gives us equivalent sufficient conditions on $T$ for $R[[x^T]]$ to be a submonoid of $R[[x^{\mathbf{N}^{*}}]]$.

\begin{theorem} If $T \subset \mathbf{N}^*$, then the following are equivalent:
\begin{description}
  \item[a)] $R[[x^T]]$ is a submonoid of $R[[x^{\mathbf{N}^{*}}]]$ for any ring $R$.
  \item[b)] $1 \in T$ and if $s,t_{1},t_{2}, ... , t_{k} \in T$ then for any partition $s=s_1+s_2+...+s_k$ we have $s_{1}t_{1}+...+s_{k}t_{k} \in T$.
  \item[c)] $1 \in T$ and if $s,t \in T$ then $s+t-1 \in T$.
\end{description}
\label{main}\end{theorem}
\begin{proof} (a $ \Rightarrow $ b) As we noted earlier, $x \in R[[x^T]]$ so we must have $1 \in T$.

 Let $R=\mathbb{Z}$ and $s,t_{1},t_{2}, ... , t_{k} \in T$. Then $x^{t_1}+...+x^{t_{k}} $, $x^{s}$ are in $R[[x^{T}]]$ and by our assumption so is $x^{s} \circ (x^{t_1}+...+x^{t_{k}} )= (x^{t_1}+...+x^{t_{k}} )^s$. When we expand $(x^{t_1}+...+x^{t_{k}} )^s$ using multinomial theorem we see that all coefficients are non-negative integers. One such coefficient is $\binom{s}{s_1,...,s_k} $ where $s=s_1+...+s_k$, i.e. the coefficient of $x^{s_1 t_1  +...+s_k t_k}$ in the expansion. So when we collect like terms the coefficient of $x^{s_1 t_1 +...+s_k t_k}$ is a strictly positive integer. That is, $s_1 t_1  +...+s_k t_k$ must be in $T$.

 (b $ \Rightarrow $ a) Since $1 \in T$ we have $x \in R[[x^T]]$ for any $R$. So we only need to show closure of $R[[x^T]]$ for composition. Note that $R[[x^T]]$ is closed under addition, multiplication by $a \in R$ and taking limits, we will have closure for composition if we can show that $x^s \circ ( a_1 x^{t_1}+...+a_k x^{t_{k}})= (a_1x^{t_1}+...+a_k x^{t_{k}})^s  \in R[[x^{T}]]$ whenever $s,t_1,...,t_k \in T$ for some $k\geq 1 $ and $a_1,...,a_k \in R$. Now it is easy to see that if we multiply $(a_1x^{t_1}+...+a_k x^{t_{k}})^s $ out and collect like terms, the only possible powers of $x$ with non-zero coefficient are those of the form $s_1 t_1  +...+s_k t_k$ for some partition $s=s_1+...+s_k$ of $s$.

  (b $ \Rightarrow $ c) By assumption $1 \in T$. Let $s,t$ be in T. Then $t+s-1=1\cdot t + \overbrace{ 1 \cdot 1 + ... + 1 \cdot 1}^{s-1}$. So we use the partition $s=1+...+1$ and the fact that $t,1,...,1 \in T$ to get $t+s-1$ is in $T$.

   (c $ \Rightarrow $ b) Again we get $1 \in T$. Given $s,t_1,...,t_k \in T$ and any partition $s=s_1+...+s_k$, we rewrite $s_1 t_1+...+s_k t_k = (s_1+...+s_k)+s_1(t_1-1)+...+s_k(t_k-1)=s+s_1(t_1-1)+...+s_k(t_k-1)$ and a repeated application of (c) s-many times gives the desired result.
  \end{proof}

\begin{definition} A subset $T \subset \mathbf{N^{*}}$ satisfying the equivalent conditions of Theorem~\ref{main} is said to be a \emph{strongly closed submonoid} of $\mathbf{N}^{*}$ (or simply strongly closed).\end{definition}

\section{Properties of Strongly Closed Submonoids}

In this section we investigate properties of strongly closed submonoids and the relationship between the strongly closed submonoids of $\mathbf{N}^{*}$ and the submonoids of $\mathbf{N}$ with respect to addition.

 \begin{proposition} The subset $T \subset \mathbf{N^{*}}$ is strongly closed if and only if $S=\{t-1 \, | \, t \in T \} $ is a submonoid of $\mathbf{N}$ (with respect to addition). \label{bijcor}
 \end{proposition}
 \begin{proof} Let $T \subset \mathbf{N^{*}}$ be strongly closed. Since $1 \in T$ we have $1-1=0 \in S$. If $s_1=t_1-1, \, s_2=t_2-1$ are in $S$, where $t_1,t_2 \in T$, then notice that $\underbrace{t_1+t_2-1}_{\in T}-1=t_1-1+t_2-1=s_1+s_2 \in S$, that is $S$ is a submonoid of $\mathbf{N}$.

 Now let $S \subset \mathbf{N}$ be a submonoid with respect to addition and let $T=\{1+s \, | \, s \in S\}$. Then since $0 \in S$ we have $1+0=1 \in T$. If $t_1=1+s_1, t_2=1+s_2 \in T$, where $s_1,s_2 \in S$, then $t_1+t_2-1=1+\underbrace{s_1+s_2}_{\in S} \in T$.
 \end{proof}

By proposition~\ref{bijcor}, we conclude that there is a bijective correspondence between strongly closed submonoids of $\mathbf{N}^{*}$ and submonoids of $(\mathbf{N},+)$. This characterization leads us to the next result which shows that strongly closed submonoids are always finitely generated.
\begin{definition} Let $X \subset \mathbf{N}^{*}$, then the strong submonoid generated by $X$ is defined to be the smallest strongly closed submonoid containing $X$.
\end{definition}

\begin{proposition}Let $X \subset \mathbf{N}^{*}$, then the strong submonoid generated by $X$ has the form
$$T=\{1+\displaystyle\sum_{a\in X} u_a(a-1) \, | \, u_a \geq 0 \text{ and } u_a=0 \text{ for all but finitely many a's} \}$$
\end{proposition}

\begin{proof} First, let us prove that $T$ is strongly closed. Clearly $1 \in T$ and $X \subset T$. Suppose $u=1+\displaystyle\sum_{a\in X} u_a(a-1)$ and $v=1+\displaystyle\sum_{a\in X} v_a(a-1)$ are in $T$, then $u+v-1=1+\displaystyle\sum_{a\in X} (u_a+v_a)(a-1) \in T$ since only finitely many of $u_a,v_a's$ are non-zero and $u+v-1$ is in the desired form.

Now we need to prove that $T$ is the smallest strongly closed submonoid containing $X$. Let $\tilde{T}$ be another strongly closed submonoid containing $X$. Notice that by successive applications of part (c) of Theorem~\ref{main} we get,
$$a_1,...,a_k \in \tilde{T} \, \Rightarrow \, a_1+a_2+...+a_k-(k-1)\in \tilde{T}$$
So we conclude that for any non-negative integer $u_a$ and $a \in X$,  $1+\displaystyle\sum_{a\in X} u_a(a-1)$ is in $\tilde{T}$.

\end{proof}

\begin{corollary} There is a bijective correspondence between the set of submonoids $S \subset \mathbf{N}$ (with respect to addition) and the strong submonoids $T \subset \mathbf{N}^{*}$. The correspondence is given by $T=1+S$.
 Moreover, if $T$ is generated by $X$ then $S$ is generated by the set $Y=X-1=\{a-1 \, | \, a \in X\}$.\label{bijgen}
\end{corollary}
\begin{proof} The result follows from the proposition~\ref{bijcor}.
\end{proof}

\begin{corollary} Every strongly closed submonoid $T\subset \mathbf{N}^*$ is finitely generated as a closed submonoid.
\end{corollary}
\begin{proof} The result follows from corollary~\ref{bijgen} and the fact that every submonoid of $(\mathbf{N}, +)$ is finitely generated. Hence, by the bijective correspondence one can get finite generating sets for the corresponding strongly closed submonoids.
\end{proof}

Surprisingly enough a strongly closed submonoid $T \subset \mathbf{N}^*$ turns out to be an infinitely generated monoid with respect to multiplication.

\begin{proposition} Let $T\subset \mathbf{N}^*$ be a strongly closed submonoid, then $T$ is an infinitely generated monoid with mutliplication as the operation.
\end{proposition}

\begin{proof} Assume that $T$ is a strongly closed submonoid and let $a \in T$ be an element different from $1$. So $a >1$, then $a+k(a-1) \in T$
 for any $k \geq 0$. But $\gcd(a,a-1)=1$, and by Dirichlet's theorem (see \cite{dir}) there are infinitely many primes among $a+k(a-1) $. So any such prime must be in the set of generators of $T$ as a monoid (w.r.t. multiplication). We conclude that there is no finite generating set of $T$ w.r.t. multiplication. \end{proof}

\section{Generalization to $n$-variable case}

In this section we consider our problem in higher dimension. Let $n \geq 2$ and consider the ring $R[[x_1,...,x_n]]$ of formal power series in $x_1,x_2...,x_n$ with coefficients in $R$. If we consider n-tuples $f=(f_1,...,f_n)$ where $f_i \in R[[x_1,...,x_n]]$ %with zero constant terms,%
 and we define composition by

\begin{equation}\label{1}
(f_1,...,f_n) \circ (g_1,...,g_n)=(f_1(g_1,...,g_n),...,f_n(g_1,...,g_n))
\end{equation}
then we get a monoid $R[[x_1,...,x_n]]^n$ with the identity $(x_1,...,x_n)$.

For any $u=(u_1,...,u_n) \in \mathbf{N}^{n}$, we denote $x^{u}=x_{1}^{u_1}x_{2}^{u_2}...x_{n}^{u_n}$. Our problem in this section will be to find the subsets $U \subset (\mathbf{N}^{n})^*=\mathbf{N}^{n}-(0,...,0)$ such that $f=(f_1,...,f_n)$ with $f_i \in R[[x^U]]$ (the formal power series in $x^{u}$ with $u \in U$) form a submonoid of $R[[x^{(\mathbf{N}^{n})^*}]]^n$. Since $(x_1,...,x_n)$ is the identity with respect to multiplication in n-dimension, we want each $x_i$ to be in $R[[x^U]]$ for all $i=1,...,n$. Note that with the previous notation $x^{e_i}=x_i$ where $e_i=(0,...,0,1,0,...,0)$ with only non-zero entry at the i-th slot, then each $e_i$ must be in $U$ for all $i=1,...,n$.

For $u=(u_1,...,u_n)$ in $\mathbf{N}^n$, we define its norm as $|u|=u_1+...+u_n$. Note that for $u,v \in \mathbf{N}^n$ , $|u+v|=|u|+|v|$. Now we are ready to give the analogous result for the n-variable case.

\begin{proposition} Let $R[[x^U]]^{n}$ be a submonoid of $R[[x^{(\mathbf{N}^{n})^*}]]^n$ with the composition operation defined as in ($1$). If $u \in U$ and $v \in (\mathbf{N}^{n})^*$ such that $|v|=|u|$, then $v\in U$. \label{abs}
\end{proposition}
\begin{proof} First we conclude that $x_1,...,x_n,$  $g=x_1+...+x_n \in R[[x^U]]$. Let $u=(u_1,...,u_n) \in U$, then $f=x_1^{u_1}...x_n^{u_n} \in R[[x^U]]$  and since $R[[x^U]]^{n}$ is a submonoid of $R[[x^{(\mathbf{N}^{n})^*}]]^n$ we conclude,
$$f \circ \overbrace{(g,...,g)}^{n-tuple}=f(g,...,g)=g^{u_1}...g^{u_n}=g^{u_1+...+u_n} \in R[[x^U]]$$
So we need $g^{|u|}=(x_1+...+x_n)^{|u|}$ to be in $R[[x^U]]$ for each $R$. In particular, if $R=\mathbb{Z}$, and $|u|=v_1+...+v_n$ is another partition of $|u|$ then using  the multinomial theorem on $(x_1+...+x_n)^{|u|}$ we get the term $\binom{|u|}{v_1,...,v_n} x_1^{v_1}...x_n^{v_n}$. Since all the coefficients in this expansion of $(x_1+...+x_n)^{|u|}$ are non-negative integers, we see that $v=(v_1,...,v_n)$ must be in $U$.
\end{proof}

Furthermore, there is an intrinsic relationship between the set $U$ of n-tuples and strongly closed submonoids as explained in details (theorem~\ref{n-tuple}). But we need following lemma in order to prove this relationship.

\begin{lemma}Let $T$ be a strongly closed submonoid of $\mathbf{N}^*$ and $U=\{u \in (\mathbf{N}^{n})^* \, | \, |u| \in T\}$. Then $f(g_1,...,g_n) \in R[[x^U]]$ for all $f,g_1,g_2,...,g_n \in R[[x^{U}]]$ if $(x^{u_1}+...+x^{u_m})^{|v|} \in \mathbb{Z}[[x^U]]$ for all $u_1,...,u_m,v \in U$.\label{reduced}
\end{lemma}
\begin{proof} Let $f,g_1,g_2,...,g_n \in R[[x^{U}]]$, we will show that $f(g_1,...,g_n) \in R[[x^U]]$. It suffices to argue this holds for $f=x^{v}$ for some $v=(v_1,...,v_n) \in U$ as $R[[x^{U}]]$ is closed under addition, scalar multiplication and taking limits. So we want $g_1^{v_1}...g_n^{v_n} \in R[[x^{U}]]$. Since $R[[x^{U}]]$ is closed under limits, we may assume that $g_i$'s are polynomials in $R[[x^U]]$ and say a linear combination of $x^{u_1},...,x^{u_m}$ for some $u_1,...,u_m \in U$ with coefficients in $R$. Let $g_i=a_1^{i}x^{u_1}+...+a_m^{i}x^{u_m}$ for $i=1,...,n$.

Note that if the coefficient of $x^{w}$ in the product $g_1^{v_1}...g_n^{v_n}$ is zero, then $w$ does not have to be in $U$. Now assume that the coefficient of $x^{w}$ is non-zero. But then we claim that the coefficient of $x^w$ in $(x^{u_1}+...+x^{u_m})^{|v|} \in \mathbb{Z}[[x]]$ is non-zero as well. This follows by the multinomial theorem because if the coefficient of $x^{w}$ is non-zero in the product $g_1^{v_1}...g_n^{v_n}$  then there must be a partition of $v$ say $|v|=t_1+...+t_m$, so that $w=u_1t_1+...+u_mt_m$. But this means $(x^{u_1}+...+x^{u_m})^{|v|}=\binom{|v|}{t_1,...,t_m} x^{w}+... $ in $\mathbb{Z}[[x^U]]$ and that the coefficient of $x^w$ is non-zero in  $(x^{u_1}+...+x^{u_m})^{|v|}$. Hence $w \in U$.
\end{proof}

\begin{theorem} $R[[x^U]]^{n}$ is a submonoid of $R[[x^{(\mathbf{N}^{n})^*}]]^n$ with the composition operation given as in ($1$)if and only if the set defined by $T=\{|u| \, | \, u \in U\}$ is a strongly closed submonoid of $\mathbf{N}^*$. \label{n-tuple}
\end{theorem}
\begin{proof}
First, let us assume that $U \subset (\mathbf{N}^{n})^*$ is such that $R[[x^U]]^{n}$ is a submonoid of $R[[x^{(\mathbf{N}^{n})^*}]]^n$. Since $e_1 \in U$, we have $|e_1|=1 \in T$. Now let $|u|,|v| \in T$ where $u,v \in U$. Since $(|u|,0,...,0)$ has the same absolute value as $u$, by proposition~\ref{abs} we conclude that $(|u|,0,...,0) \in U$. Similarly $(|v|,0,...,0)\in U$. So $x^{e_1}+x^{(|u|,0,...,0)}=x_1+x_1^{|u|} \in R[[x^U]]$. Substituting, $x_1+x_1^{|u|}$ for $x$, in $x^{(|v|,0,...,0)}=x_1^{|v|}$ we get,
\begin{eqnarray*}
(x_1+x_1^{|u|})^{|v|} &=& x_1^{|v|}+|v|x_1^{|v|-1}x_1^{|u|}+...\\
 &=& x_1^{|v|}+|v|x_1^{|u|+|v|-1}+...
\end{eqnarray*}
Letting $R=\mathbb{Z}$ we see that we must have $(|u|+|v|-1,0,...,0) \in U$. So its absolute value, i.e. $|u|+|v|-1$, must be in $T$ and $T$ is strongly closed.

To show the converse, assume that $T \subset \mathbf{N}^*$ is a strongly closed submonoid. Let $U=\{u \in (\mathbf{N}^{n})^* \, | \, |u| \in T \}$. Since $1 \in T$ and since $|e_i|=1$ for $0 \leq i \leq n$, we have $e_1,...,e_n \in U$. So $x^{e_i}=x_i \in R[[x^U]]$ for $i=1,...,n$.

By lemma~\ref{reduced}, $f(g_1,...,g_n) \in R[[x^U]]$ for all $f,g_1,g_2,...,g_n \in R[[x^{U}]]$ if $(x^{u_1}+...+x^{u_m})^{|v|} \in \mathbb{Z}[[x^U]]$ for all $u_1,...,u_m,v \in U$. So we just need to prove that  $(x^{u_1}+...+x^{u_m})^{|v|} \in \mathbb{Z}[[x^U]]$ whenever $u_1,...,u_m,v \in U$. Assume we are given  $u_1,...,u_m,v \in U$ by multionomial theorem, it suffices to show that $s_1 u_1+...+s_m u_m \in U$ for any partition $|v|=s_1+...+s_m$. Notice that, this is the case if and only if $|s_1u_1+...+s_m u_m |=s_1|u_1|+...+s_m|u_m| \in T$, but that holds as $T$ is strongly closed. So we conclude that $f(g_1,...,g_n) \in R[[x^U]]$.\end{proof}

\section{The invertible elements of $R[[x^{T}]]$}
We turn our attention to the invertible elements of $R[[x]]$ with respect to composition. Throughout this section, invertibility will be understood with respect to composition. It is easy to see that a formal power series $f$ is invertible if and only if its constant term is zero and $f'(0)$ is invertible in $R$. We show that for an arbitrary $T \subseteq \mathbf{N}^{*}$ the subset consisting of the invertible elements of $ R[[x^{T}]]$  forms a subgroup if and only if $T$ is strongly closed.

\begin{theorem}Let $T \subseteq \mathbf{N}^*$. Then the subset $H \subset R[[x^{T}]]$ of invertible elements forms a group with composition as the operation for any ring $R$ if and only if $T$ is strongly closed.
\end{theorem}
\begin{proof} Assume that the subset $H \subset R[[x^{T}]]$ of invertible elements forms a group. Then $x \in H$ implies $1 \in T$. So we only need to show that if $s,t \in T$ then $s+t-1 \in T$. This is trivial if $s=1$ or $t=1$, so assume $s,t \geq 2$. Let $R=\mathbb{Z}$, then both $x+x^t, x+x^s \in \mathbb{Z}[[x^T]]$ are invertible and,
 $$(x+x^s)\circ (x+x^t)= x+x^t + (x+x^t)^s \in H \subseteq \mathbb{Z}[[x^T]]$$
We have,
$$ (x+x^t)^s=x^s +\binom{s}{1} x^{s-1+t}+\binom{s}{2} x^{s-2+2t}+ \, ...$$
Note that since $s,t \geq 2$ we have $s < s-1+t < s-2+2t < \, ...$. So the exponent $s+t-1$ appears in $(x+x^t)^s$. It doesn't appear in $x+x^t$, so it appear in $x+x^t + (x+x^t)^s$, hence $s+t-1 \in T$.

Now we assume that $T$ is strongly closed. It is enough to show that if $f \in R[[x^{T}]]$ is invertible then $f^{-1}$ is in $R[[x^{T}]]$ as well. Let $T=\{1=t_1 < t_2 <t_3 < ...\}$ and $f(x)=\displaystyle\sum_{t \in T} a_t x^t$ be invertible. Say $f^{-1}(x)=\displaystyle\sum_{n=1}^{\infty}b_n x^n$ , we will show inductively that if $b_n \neq 0$ then $n \in T$. For $n=1$ there is nothing to show since $1 \in T$. So suppose for each of $b_1,...,b_{n-1}$ that either it is $0$ or that when $b_{k} \neq 0$ then $k \in T$ for $1 \leq k \leq n-1$. We have $ f(x) \circ f^{-1}(x)=x$ truncating we get,
$$f(x)\circ (b_1x+...+b_n x^{n}) \cong x \, (\text{mod } x^{n+1}) $$
That is,
$$a_1(b_1x+...+b_n x^{n})+a_{t_{2}} (b_1x+...+b_n x^{n})^{t_2}+...\cong x \, (\text{mod } x^{n+1})   $$
But note that $a_{t_{2}} (b_1x+...+b_n x^{n})^{t_2} \cong a_{t_{2}} (b_1x+...+b_{n-1} x^{n-1})^{t_2} \, (\text{mod } x^{n}) $ since $n \geq 2$.Likewise for $a_{t_{3}} (b_1x+...+b_n x^{n})^{t_3}$ etc. So we have,
$$ a_1(b_1x+...+b_n x^{n})+\underbrace{a_{t_{2}} (b_1x+...+b_{n-1} x^{n-1})^{t_2}+a_{t_{3}} (b_1x+...+b_{n-1} x^{n-1})^{t_3}+ ...}_{\ast}\cong x \, (\text{mod } x^{n+1}) $$
By induction hypothesis with $n \geq 2$, $(b_1x+...+b_{n-1} x^{n-1})^{t_2}$ and $(b_1x+...+b_{n-1} x^{n-1})^{t_3}$ are in $R[[x^T]]$. So $\ast \in R[[x^T]]$.

Now investigating the coefficient of $x^n$ in $a_1(b_1x+...+b_n x^{n})+\ast$, we see that it is $a_1b_n+\text{ the coefficient of } x_n \text{ in } \ast$. But this coefficient must be zero since,
$$a_1(b_1x+...+b_n x^{n})+\ast \cong x  (\text{ mod } x^{n+1}) $$
and since $n \geq 2$.

So if $b_n \neq 0$, then $a_1b_n \neq 0$ since $a_1$ is a unit. That is, $a_1b_n=-\text{ the coefficient of } x_n \text{ in } \ast $ which is then not zero. Now, since $\ast \in R[[x^T]]$ we conclude that $n\in T$. This concludes our induction, hence we get $f^{-1} \in R[[x^T]]$.
\end{proof}

\end{document}